\documentclass{article}
\usepackage[utf8]{inputenc}
\usepackage{authblk}
\usepackage{setspace}
\usepackage[margin=1.25in]{geometry}
\usepackage{graphicx}
\graphicspath{ {./figures/} }
\usepackage{subcaption}
\usepackage{amsmath}
\usepackage{amssymb}
\usepackage{amsthm}
\usepackage{tikz}
\usepackage{tikz-cd}
\usepackage[utf8]{inputenc}
\usepackage{bbold}
\usepackage[hidelinks]{hyperref}      
\usepackage[nameinlink]{cleveref}   
\theoremstyle{plain}
\newtheorem{theorem}{Theorem}[section]
\newtheorem{lemma}[theorem]{Lemma}

\newtheorem{corollary}[theorem]{Corollary}
\newtheorem{claim}[theorem]{Claim}

\theoremstyle{definition}
\newtheorem{definition}[theorem]{Definition}
\newtheorem{example}[theorem]{Example}

\theoremstyle{remark}
\newtheorem{remark}[theorem]{Remark}

\newtheorem*{notation}{Notation}
\usepackage[backend=bibtex, style=alphabetic, citestyle=alphabetic, sorting=ynt]{biblatex}
\bibliography{sample}{}

\newcommand{\N}{\mathbb{N}}

\newcommand{\di}{dimension}
\newcommand{\floor}[1]{\left\lfloor #1 \right\rfloor}
\newcommand{\ceil}[1]{\left\lceil #1 \right\rceil}
\newcommand{\prs}[1]{\left( #1 \right)}
\newcommand{\crb}[1]{\left\{ #1 \right\}}

\usetikzlibrary{patterns}
\tikzset{
pattern size/.store in=\mcSize, 
pattern size = 5pt,
pattern thickness/.store in=\mcThickness, 
pattern thickness = 0.3pt,
pattern radius/.store in=\mcRadius, 
pattern radius = 1pt}
\makeatletter
\pgfutil@ifundefined{pgf@pattern@name@redstripes}{
\pgfdeclarepatternformonly[\mcThickness,\mcSize]{redstripes}
{\pgfqpoint{0pt}{-\mcThickness}}
{\pgfpoint{\mcSize}{\mcSize}}
{\pgfpoint{\mcSize}{\mcSize}}
{
\pgfsetcolor{\tikz@pattern@color}
\pgfsetlinewidth{\mcThickness}
\pgfpathmoveto{\pgfqpoint{0pt}{\mcSize}}
\pgfpathlineto{\pgfpoint{\mcSize+\mcThickness}{-\mcThickness}}
\pgfusepath{stroke}
}}
\makeatother
 
\tikzset{
pattern size/.store in=\mcSize, 
pattern size = 5pt,
pattern thickness/.store in=\mcThickness, 
pattern thickness = 0.3pt,
pattern radius/.store in=\mcRadius, 
pattern radius = 1pt}
\makeatletter
\pgfutil@ifundefined{pgf@pattern@name@_7la7libyc}{
\pgfdeclarepatternformonly[\mcThickness,\mcSize]{_7la7libyc}
{\pgfqpoint{0pt}{0pt}}
{\pgfpoint{\mcSize+\mcThickness}{\mcSize+\mcThickness}}
{\pgfpoint{\mcSize}{\mcSize}}
{
\pgfsetcolor{\tikz@pattern@color}
\pgfsetlinewidth{\mcThickness}
\pgfpathmoveto{\pgfqpoint{0pt}{0pt}}
\pgfpathlineto{\pgfpoint{\mcSize+\mcThickness}{\mcSize+\mcThickness}}
\pgfusepath{stroke}
}}
\makeatother
\tikzset{every picture/.style={line width=0.75pt}}

\title{Vertex-Minimal Triangulation of Complexes with Homology}
\author{Jon V. Kogan\footnote{Department of Mathematics, Hebrew University, Jerusalem 91904, Israel. e-mail:
jonatan.kogan$@$mail.huji.ac.il. Partially supported by ERC grant 3012006831.}}
\date{\vspace{-5ex}}

\begin{document}
\maketitle
\begin{abstract}
For a given pair of numbers $(d,k)$, we establish the minimal number of vertices in pure $d$-\di al simplicial complexes with non-trivial homology in \di \ $k$. Furthermore, we solve the problem under the additional constraint of strong connectivity with respect to any intermediate \di. 
\end{abstract}
\section{Introduction}\label{intro}
Simplicial complexes are a natural generalization of graphs and have been studied extensively in recent decades- whether as clique complexes of graphs, triangulations of manifolds, or random models for topological spaces.
Given the strong interest in extremal graph theory, it is natural to ask extremal questions in the context of simplicial complexes, specifically those that involve the number of faces within these complexes.

A central result in this area is the Kruskal–Katona theorem (see \cite{Kruskal63}), which provides a description of the possible face-vectors of a simplicial complex.
Subsequently, Björner and Kalai related matters to topology in  \cite{bjornerkalai88}, by showing the relationship between the face vector and Betti vector of a simplicial complex.
\par
In this paper, we alter the question in \cite{bjornerkalai88} by adding 
geometric requirements that the simplicial complex must satisfy.
For instance, it is plain to see that the smallest number of vertices of a complex with non-trivial $k$-homology is $k+2$ (in $\partial(\Delta^{k+1})$), and for a pure $d$-\di al complex the minimum is $d+1$.
Combining these requirements turns out to result in a more interesting problem.

In \Cref{pure} we prove:
\begin{theorem}\label{thmPure}
    Let $0\le k\le d$. Any pure $d$-dimensional simplicial complex with nontrivial $H_k$ has at least $\ceil{\frac{(d+1)(k+2)}{k+1}}$ vertices, and this bound is tight.
\end{theorem}

Additionally, we explore strong connectivity (defined in \Cref{defStrongC}), an equivalence relation on simplexes, whose components are crucial for understanding the cohomological structure of the entire complex, as well as cohomology operations such as the cup product and Steenrod operations (see \Cref{lemCupSteenConc}).
We prove the following result:

\begin{theorem}\label{thmStr}
    Let $1\le k\le d$. Any strongly-connected $d$-dimensional complex with non-trivial $H_k$ has at least $d+1+\ceil{\frac{d}{k}}=\ceil{\frac{d(k+1)}{k}+1}$ vertices, and this bound is tight.
\end{theorem}
This is a special case of a more general theorem on relative strong connectivity, proven in \Cref{secRel}:
\begin{theorem}\label{thmRel}
    Let $1\le k\le d$ and $X$ a pure $d$-\di al complex, strongly connected w.r.t \di \ $m>d+1 -\ceil{\frac{d+1}{k+1}}$, such that $H_k(X)\ne 0$. Then $X$ has at least $d+1+\ceil{\frac{m}{k}}$ vertices, and this bound is tight.
\end{theorem}
The importance of strongly connected components w.r.t \di \ $m$ stems essentially from the fact that these components determine collapsability to \di\ $m-1$ (see \Cref{lemColapseComponents}).

Finally, in \Cref{secDiscuss} we relate these results to other work in the field.

\section{Preliminaries}\label{prelim}
\begin{notation}
We use the following conventions:
\begin{itemize}
    \item By \textbf{simplicial complex} (or simply \textbf{complex}), we mean both the combinatorial object called "abstract simplicial complex" and the CW complex resulting from its geometric realization. The distinction will not be relevant in this paper.
    \item By $H_k$ we mean simplicial $k$-homology with coefficients in any abelian group.
    \item For a simplicial complex $X$, $|X|$ denotes the number of vertices in $X$.
\end{itemize}
\end{notation}
\begin{definition}
    A simplicial complex is called \textbf{pure $d$-dimensional}  if any simplex in the complex is contained in a $d$ \di al face. Note that this implies that $d$ is the dimension of all maximal faces.
\end{definition}
\begin{definition}\label{defStrongC}
 For a simplicial complex $C$, define a relation $\sim$ on its $d$-dimensional faces by having $\sigma \sim \tau$ if they share a $d-1$-dimensional face. 
 This is obviously reflexive and symmetric. Complete this relation under transitivity and call it $\sim'$. A \textbf{strong connectivity component} or \textbf{strongly connected component} of $C$ is an equivalence class of $d$-dimensional faces of $C$ under this relation. 
 A pure-$d$-dimensional simplicial complex is called \textbf{Strongly Connected} if all of its $d$-\di al faces are in the same component.
 
 A simplicial complex where all maximal simplexes are of \di \ greater or equal to $d$ is called \textbf{Strongly Connected with respect to \di \ $d$} if its $d$-skeleton is strongly connected.
 Equivalently, a complex is strongly connected w.r.t \di \ $d$ if its maximal faces form a single component under the relation of sharing $d-1$ \di al faces.
\end{definition}

\begin{definition}\label{expansionOp}
    For a simplicial complex $X$ strongly connected w.r.t \di \ $d$, an \textbf{expansion operation} is the attachment $X\cup _A \Delta^m$ of a new face to $X$, such that $m\ge d$ and $A$ contains a $(d-1)$-face (thus the new complex is also strongly connected w.r.t  \di \ $d$).
\end{definition} 
In this paper, we will discuss strongly connected components as resulting from expansion operations. Expansion operations fit into a \textbf{growth process}, i.e a series of subcomplexes $A_0\subset A_1\subset ...A_k=A$, where $A_0$ is a simplex and $\dim(A_0)\ge d$, $A$ is any strongly connected $d$-complex,  $A_{i+1}$ has exactly one maximal face more than $A_i$ and all $\{A_i\}$ are strongly connected w.r.t \di \ $d$.

\begin{claim}[from the preprint \cite{kogan24}]\label{vertBeforeAles}
 For any $C$ strongly connected w.r.t \di\ $d$, there exists a growth process.
\end{claim}
\begin{proof}
Define a graph $G_C=(V_C,E_C)$, where $V_C$ is the set of maximal faces of $C$, and $(\mu,M)$ is an edge if and only if $\dim(\mu\cap M)\ge d-1$. 
    The strong connectivity w.r.t \di\ $d$ of $C$ is equivalent to $G_C$ being connected, and so $G_C$ contains a spanning tree.
    Any tree has a filtration $\{v\}=T_0\subset T_1\subset ...\subset T_m=T$, where we add 1 vertex and 1 edge at a time, keeping the tree connected, and such a filtration corresponds to the desired $\Delta^d=C_0\subset C_1\subset...\subset C_m=C$.

\end{proof}

\begin{example}\label{expDim2}
    
 The following are the expansion operations for the pure 2-dimensional case:
\begin{center}
\begin{tabular}{| c  | c |} 
 \hline
 Operation & Homotopical effect\\
 \hline
\begin{tikzpicture}[x=0.75pt,y=0.75pt,yscale=-1,xscale=1]

\draw  [fill={rgb, 255:red, 155; green, 155; blue, 155 }  ,fill opacity=1 ] (16.25,54.25) -- (56.25,34.25) -- (56.25,74.25) -- cycle ;
\draw  [color={rgb, 255:red, 208; green, 2; blue, 27 }  ,draw opacity=1 ][pattern=redstripes,pattern size=6pt,pattern thickness=0.75pt,pattern radius=0pt, pattern color={rgb, 255:red, 208; green, 2; blue, 27}] (96.25,54.25) -- (56.25,74.25) -- (56.25,34.25) -- cycle ;
\draw  [fill={rgb, 255:red, 208; green, 2; blue, 27 }  ,fill opacity=1 ] (92,54.25) .. controls (92,51.9) and (93.9,50) .. (96.25,50) .. controls (98.6,50) and (100.5,51.9) .. (100.5,54.25) .. controls (100.5,56.6) and (98.6,58.5) .. (96.25,58.5) .. controls (93.9,58.5) and (92,56.6) .. (92,54.25) -- cycle ;
\end{tikzpicture} &   

    None \\
 \hline 
\begin{tikzpicture}[x=0.75pt,y=0.75pt,yscale=-1,xscale=1]

\draw  [fill={rgb, 255:red, 155; green, 155; blue, 155 }  ,fill opacity=1 ] (16.25,54.25) -- (56.25,34.25) -- (56.25,74.25) -- cycle ;
\draw  [color={rgb, 255:red, 208; green, 2; blue, 27 }  ,draw opacity=1 ][pattern=redstripes,pattern size=6pt,pattern thickness=0.75pt,pattern radius=0pt, pattern color={rgb, 255:red, 208; green, 2; blue, 27}] (96.25,54.25) -- (56.25,74.25) -- (56.25,34.25) -- cycle ;
\draw  [fill={rgb, 255:red, 0; green, 0; blue, 0 }  ,fill opacity=1 ] (92,54.25) .. controls (92,51.9) and (93.9,50) .. (96.25,50) .. controls (98.6,50) and (100.5,51.9) .. (100.5,54.25) .. controls (100.5,56.6) and (98.6,58.5) .. (96.25,58.5) .. controls (93.9,58.5) and (92,56.6) .. (92,54.25) -- cycle ;
\end{tikzpicture} & Adds generator to $H^1$\\
 \hline 

\begin{tikzpicture}[x=0.75pt,y=0.75pt,yscale=-1,xscale=1,rotate=-90]
\draw[opacity=0] (67,127) -- (28,127);
\draw  [color={rgb, 255:red, 208; green, 2; blue, 27 }  ,draw opacity=1 ][pattern=redstripes,pattern size=6pt,pattern thickness=0.75pt,pattern radius=0pt, pattern color={rgb, 255:red, 208; green, 2; blue, 27}] (28.93,81.57) -- (63.57,61.57) -- (63.57,101.57) -- cycle ;
\draw  [fill={rgb, 255:red, 155; green, 155; blue, 155 }  ,fill opacity=1 ] (26.25,36.93) -- (63.57,61.57) -- (28.93,81.57) -- cycle ;
\draw  [fill={rgb, 255:red, 155; green, 155; blue, 155 }  ,fill opacity=1 ] (26.25,126.21) -- (28.93,81.57) -- (63.57,101.57) -- cycle ;
\end{tikzpicture} 
& 
    None
\\
 \hline 
 \begin{tikzpicture}[x=0.75pt,y=0.75pt,yscale=-0.75,xscale=0.75]

\draw  [color={rgb, 255:red, 208; green, 2; blue, 27 }  ,draw opacity=1 ][pattern=redstripes,pattern size=6pt,pattern thickness=0.75pt,pattern radius=0pt, pattern color={rgb, 255:red, 208; green, 2; blue, 27}] (28.93,81.57) -- (63.57,61.57) -- (63.57,101.57) -- cycle ;
\draw  [fill={rgb, 255:red, 155; green, 155; blue, 155 }  ,fill opacity=1 ] (26.25,36.93) -- (63.57,61.57) -- (28.93,81.57) -- cycle ;
\draw  [fill={rgb, 255:red, 155; green, 155; blue, 155 }  ,fill opacity=1 ] (26.25,126.21) -- (28.93,81.57) -- (63.57,101.57) -- cycle ;
\draw  [fill={rgb, 255:red, 155; green, 155; blue, 155 }  ,fill opacity=1 ] (103.57,81.57) -- (63.57,101.57) -- (63.57,61.57) -- cycle ;

\end{tikzpicture} & Reduces $H^1$ or increases $H^2$ \\
 \hline
\end{tabular}

\end{center}\par
The black and gray simplexes denote the complex before the expansion operation,  while the red patterned simplexes denote new faces. \par
The following is a growth process for $C=\partial(\Delta^3)$: 
$$\begin{tikzpicture}

\draw[fill=gray,fill opacity=0.4]  (0,0) 
  -- (2,0) 
  -- (2.2,1) 
  -- cycle;

\draw (0,0) node[left]{$A$}
(2,0) node[right]{$B$}
  (2.2,1) node[right]{$C$};
\end{tikzpicture}
\begin{tikzpicture}

\draw[fill=gray,fill opacity=0.4]  (0,0) 
  -- (2,0) 
  -- (2.2,1) 
  -- cycle;
\draw[pattern=redstripes,pattern size=6pt,pattern thickness=0.75pt,pattern radius=0pt, pattern color={rgb, 255:red, 208; green, 2; blue, 27}]  (0,0) 
  -- (1,1.6) 
  -- (2.2,1) 
  -- cycle;
\draw (0,0) node[left]{$A$}
(1,1.6) node[above]{$D$}
(2,0) node[right]{$B$}
  (2.2,1) node[right]{$C$};
\end{tikzpicture}
\begin{tikzpicture}

\draw[fill=gray,fill opacity=0.3]  (0,0) 
  -- (2,0) 
  -- (2.2,1) 
  -- cycle;
\draw[fill=gray,fill opacity=0.3]  (0,0) 
  -- (1,1.6) 
  -- (2.2,1) 
  -- cycle;
\draw (0,0) node[left]{$A$}
(1,1.6) node[above]{$D$}
(2,0) node[right]{$B$}
  (2.2,1) node[right]{$C$};
\draw[pattern=redstripes,pattern size=6pt,pattern thickness=0.75pt,pattern radius=0pt, pattern color={rgb, 255:red, 208; green, 2; blue, 27}]
(0,0) 
  -- (1,1.6) 
  -- (2,0)
  -- cycle;
\end{tikzpicture}
\begin{tikzpicture}
\draw[fill=gray,fill opacity=0.2]  (0,0) 
  -- (1,1.6) 
  -- (2,0)
  -- cycle;

\draw[fill=gray,fill opacity=0.3]  (0,0) 
  -- (2,0) 
  -- (2.2,1) 
  -- cycle;
\draw[fill=gray,fill opacity=0.3]  (0,0) 
  -- (1,1.6) 
  -- (2.2,1) 
  -- cycle;
\draw (0,0) node[left]{$A$}
(1,1.6) node[above]{$D$}
(2,0) node[right]{$B$}
  (2.2,1) node[right]{$C$};
\draw[pattern=redstripes,pattern size=6pt,pattern thickness=0.75pt,pattern radius=0pt, pattern color={rgb, 255:red, 208; green, 2; blue, 27}] (1,1.6) 
  -- (2,0) 
  -- (2.2,1) 
  -- cycle; 
\end{tikzpicture}
$$
So $C_0$ is the triangle $ABC$, and $C_1=C_0\cup _{AC} ACD$.
$C_2=C_1\cup _{AD\cup AB} ABD$ and finally $C=C_2\cup _{\partial(BCD)} BCD$.
 \end{example}
We now move on to some topological results.
\begin{theorem}[The Mayer–Vietoris sequence,{\cite[section 2.2]{AT}}]
    Let $X$ be a CW complex, and $A,B$ be subcomplexes such that $A\cup B=X$. Then there exists an exact sequence
    \[
\begin{tikzcd}
\cdots \arrow[r, ""] 
& H_n(A \cap B) \arrow[r, "i_* \oplus j_*"] 
& H_n(A) \oplus H_n(B) \arrow[r, "k_* - l_*"] 
& H_n(X) \arrow[r, "\partial"] 
& H_{n-1}(A \cap B) \arrow[r, ""] 
& \cdots
\end{tikzcd}
\]
where $i,j$ are respectively the inclusion maps from $A\cap B$ to $A,B$, and $k,l$ are respectively the inclusion maps from $A,B$ to $X$.
\end{theorem}
 In particular, we will use the following corollary:
\begin{corollary}
    In the above notation, assume in addition that $B$ is contractible.
If $H_n(A)=H_{n-1}(A)=0$, then $H_n(X)=H_{n-1}(A\cap B)$.
\end{corollary}
We also extensively use the following notion:
\begin{definition}
    For a covering of a topological space $X$ by subsets $\mathcal{U} =\crb{U_i}_{i\in I}$, the \textbf{nerve of $(X,\mathcal{U} )$} (denoted $N(\mathcal{U})$) is a simplicial complex with vertex-set $I$, where a set of vertices $\crb{i_1,...,i_n}$ comprise a simplex if and only if $\bigcap_{j=1}^n U_{i_j}\ne 0$. 
\end{definition}
\begin{theorem}[Leray's nerve theorem, {\cite[4G exercise 4]{AT}}]\label{thm:Leray}
    Let $X$ be a CW complex, and let $\mathcal{U} = \{ U_i \}_{i \in I}$ be a covering of $X$ by sub-CW complexes.
    If every finite intersection $U_{i_0} \cap \cdots \cap U_{i_k}$ is either empty or contractible, then the nerve $N(\mathcal{U})$ is homotopy equivalent to $X$.
\end{theorem}
In particular, for a simplicial complex $X$, we will denote by $N_{max}(X)$ the nerve of $X$ w.r.t the cover by maximal faces. 
    This is a cover by simplexes, thus all intersections are either empty or contractible, and so we may apply Leray's nerve theorem to conclude $N_{max}(X)\overset{h}{\sim} X$.
\begin{lemma}\label{lem:Nerve}
    A complex $X$ that has non-trivial $H_k$ has at least $k+2$ maximal faces (and this is tight). Furthermore, for any maximal face $M$, there exist maximal faces $m_2,...,m_{k+2}$ such that $M\cap \bigcap_{i=2} ^{k+2}m_i=\emptyset$.
\end{lemma}

\begin{proof}
    By assumption, \Cref{thm:Leray} entails that $H_k(N_{max}(X))\ne 0$ as well.
    We conclude that $N_{max}(X)$ has at least $k+2$ vertices, and that $sk_{k+1}(N_{max}(X))$ is not full, i.e there exist $f_0,...,f_{k+1}$ vertices in $N_{max}(C)$ that do not span a simplex. 
    By the definition of the nerve, this is equivalent to $\bigcap_{i=0} ^{k+1}f_i=\emptyset$ as maximal faces in $X$.\par
    Moreover, for a maximal face $M$, if any collection of $k+2$ maximal faces containing $M$ intersect, then $M$ corresponds to a cone-point above the $k$ skeleton of $N_{max}(X)$ ($lk(M)\supset sk_k(N_{max}(X)\setminus M)$), meaning it can not have $H_k\ne 0$.
\end{proof}

\section{Pure Dimensional Case}\label{pure}
\begin{example}\label{expure}
    Consider  $\ceil{\frac{(d+1)(k+2)}{k+1}}$ vertices, divided into $k+3$ sets $s_0,...,s_{k+1},z$. For 
    $$r=\begin{cases}
        (d+1)\%\  (k+1), & k+1 \nmid d+1 \\
        k+1, & otherwise
        \end{cases}$$
    the sets $s_0,...,s_{r}$ contain $\ceil{\frac{d+1}{k+1}}$,  $s_{r+1},...,s_{k+1}$
    contain $\floor{\frac{d+1}{k+1}}$, and $z$ contains $1$ unless $k+1|d+1$, in which case it contains 0. By "$\%$" we mean division remainder. \par
    Construct the complex $MH_{d,k}$ to have $k+2$ maximal $d$-faces $f_0,...,f_{k+1}$, where $f_i=\bigcup_{j\ne i}s_j \bigcup_{j=i}^r z$, meaning we add $z\in f_i$ if and only if $i\le r$. 
    All $f_i$ have a total of $d+1$ vertices, thus
    $MH_{d,k}$ is a pure $d$-\di al complex with $N_{max}(MH_{d,k})=\partial(\Delta^{k+1})$, and so $H_k(MH_{d,k})\ne 0$. 
    Furthermore, if $r\ne k+1$:
    $$\ceil{\frac{(d+1)(k+2)}{k+1}} = \ceil{\sum_{i=1}^{k+2} \frac{d+1}{k+1}} = (k+2)\floor{\frac{d+1}{k+1}} + \ceil{(k+2)\crb{\frac{d+1}{k+1}}} =$$
$$ (k+2)\floor{\frac{d+1}{k+1}} + \ceil{(k+2)\frac{r}{k+1}} = (k+2)\floor{\frac{d+1}{k+1}} + \ceil{r + \frac{r}{k+1}} = (k+2)\floor{\frac{d+1}{k+1}} + r + 1$$
   where $\crb{a}$ denotes the fractional part of $a$.
    Thus $MH_{d,k}$ has the claimed number of vertices, coinsiding with the bound from \ref{thmPure}, proving tightness. The case where $r=k+1$ is obvious. 
\end{example}
\begin{figure}[ht]
    \centering
\begin{tikzpicture}
\draw (0,0) -- (1,1.6);
\draw (0,0) -- (2,0);
\draw (1,1.6) -- (2,0);
\end{tikzpicture}
\begin{tikzpicture}
\draw[fill=gray,fill opacity=0.4] (0,0) -- (1,1.6) -- (1,1.9) -- cycle;
\draw[fill=gray,fill opacity=0.4] (2,0) -- (1,1.6) -- (1,1.9) -- cycle;
\draw[fill=gray,fill opacity=0.4] (0,0) -- (2,0) -- (1,-0.5) -- cycle;
\end{tikzpicture}
\begin{tikzpicture}
\draw (0,0) -- (1,1.6);
\filldraw[black] (0,0) circle (1pt) node[anchor=east]{2};
\filldraw[black] (1,1.6) circle (1pt) node[anchor=west]{2};
\filldraw[black] (2,0) circle (1pt) node[anchor=west]{2};
\draw (0,0) -- (2,0);
\draw (1,1.6) -- (2,0);
\end{tikzpicture}
\begin{tikzpicture}
\draw (0,0) -- (1,1.6);
\draw (1,1.6) -- (2,0);
\draw[fill=gray,fill opacity=0.4] (0,0) -- (2,0) -- (1,-0.5) -- cycle;
\filldraw[black] (0,0) circle (1pt) node[anchor=east]{2};
\filldraw[black] (1,1.6) circle (1pt) node[anchor=west]{3};
\filldraw[black] (2,0) circle (1pt) node[anchor=west]{2};
\end{tikzpicture}
    \caption{The first examples for when $k=1$.}
    \label{fig:k=1}
\end{figure}

\begin{figure}[ht]
    \centering
\begin{tikzpicture}
\draw[fill=gray,fill opacity=0.2]  (0,0) 
  -- (1,1.6) 
  -- (2,0)
  -- cycle;

\draw[fill=gray,fill opacity=0.3]  (0,0) 
  -- (2,0) 
  -- (2.2,1) 
  -- cycle;
\draw[fill=gray,fill opacity=0.3]  (0,0) 
  -- (1,1.6) 
  -- (2.2,1) 
  -- cycle;
\draw (0,0) node[left]{}
(1,1.6) node[above]{}
(2,0) node[right]{}
  (2.2,1) node[right]{};
\draw[fill=gray,fill opacity=0.3] (1,1.6) 
  -- (2,0) 
  -- (2.2,1) 
  -- cycle; 
\end{tikzpicture}
\begin{tikzpicture}
\draw[fill=gray,fill opacity=0.2]  (0,0) 
  -- (1,1.6) 
  -- (2,0)
  -- cycle;

\draw[fill=gray,fill opacity=0.3]  (0,0) 
  -- (2,0) 
  -- (2.2,1) 
  -- cycle;
\draw[fill=gray,fill opacity=0.3]  (0,0) 
  -- (1,1.6) 
  -- (2.2,1) 
  -- cycle;
\draw (0,0) node[left]{}
(1,1.6) node[above]{2}
(2,0) node[right]{}
  (2.2,1) node[right]{};
\draw[fill=gray,fill opacity=0.3] (1,1.6) 
  -- (2,0) 
  -- (2.2,1) 
  -- cycle; 
\draw[fill=gray,fill opacity=0.4] (0,0) 
  -- (2,0) 
  -- (1.4,-0.3) 
  -- cycle; 
\draw (2.2,1) 
  -- (2,0) 
  -- (1.4,-0.3) 
  -- cycle; 
\end{tikzpicture}
\begin{tikzpicture}
\draw[fill=gray,fill opacity=0.2]  (0,0) 
  -- (1,1.6) 
  -- (2,0)
  -- cycle;
\draw[fill=gray,fill opacity=0.3]  (0,0) 
  -- (2,0) 
  -- (2.2,1) 
  -- cycle;
\draw[fill=gray,fill opacity=0.3]  (0,0) 
  -- (1,1.6) 
  -- (2.2,1) 
  -- cycle;
\draw (0,0) node[left]{2}
(1,1.6) node[above]{2}
(2,0) node[right]{}
  (2.2,1) node[right]{};
\draw[fill=gray,fill opacity=0.3] (1,1.6) 
  -- (2,0) 
  -- (2.2,1) 
  -- cycle; 
\draw[fill=gray,fill opacity=0.3] (1,1.6) 
  -- (2,0) 
  -- (3.2,-0.3) 
  -- cycle; 
  \draw[fill=gray,fill opacity=0.3] (1,1.6) 
  -- (2.2,1) 
  -- (3.2,-0.3) 
  -- cycle; 
  \draw[fill=gray,fill opacity=0.3] (0,0) 
  -- (2,0) 
  -- (3.2,-0.3) 
  -- cycle; 
  \draw[fill=gray,fill opacity=0.3] (0,0) 
  -- (2.2,1) 
  -- (3.2,-0.3) 
  -- cycle; 
\end{tikzpicture}
\begin{tikzpicture}
\draw[fill=gray,fill opacity=0.2]  (0,0) 
  -- (1,1.6) 
  -- (2,0)
  -- cycle;

\draw[fill=gray,fill opacity=0.3]  (0,0) 
  -- (2,0) 
  -- (2.2,1) 
  -- cycle;
\draw[fill=gray,fill opacity=0.3]  (0,0) 
  -- (1,1.6) 
  -- (2.2,1) 
  -- cycle;
\draw (0,0) node[left]{2}
(1,1.6) node[above]{2}
(2,0) node[right]{2}
  (2.2,1) node[right]{2};
\draw[fill=gray,fill opacity=0.3] (1,1.6) 
  -- (2,0) 
  -- (2.2,1) 
  -- cycle; 
\end{tikzpicture}

    \caption{The first examples for when $k=2$}
    \label{fig:k=2}
\end{figure}
\begin{remark}
    In \Cref{fig:k=1} and \Cref{fig:k=2}, a number $l$ next to a vertex indicates a collection of $l$ vertices which are "the same", in the sense that any maximal face containing one of them contains all of them. 
\end{remark}

\begin{proof}[Proof of \Cref{thmPure}]
    From \Cref{lem:Nerve} we conclude that there exist $f_0,...,f_{k+1}$ maximal faces of $C$ (which are $d$-\di al) such that $\bigcap_{i=0} ^{k+1}f_i=\emptyset$.\par
    Thus each vertex of $f_0,...,f_{k+1}$ in $C$ is common to at most $k+1$ of the $k+2$ faces, and by the pigeonhole principal, $\bigcup_{i=0} ^{k+1}f_i$ must have at least $\frac{(d+1)(k+2)}{k+1}$ vertices (each face has $d+1$ vertices).\par
    \Cref{expure} proves tightness.
\end{proof}

\section{Strongly Connected Case}
    \Cref{thmStr} follows from the more general case in \Cref{thmRel}.
     Because of \Cref{lemCupSteenConc}, the strongly connected pure \di al case is of special interest, so we construct examples addressing it in particular:
\begin{example} 
    Let $V=\{v_1,...,v_d\}$, and denote $d\%  k=r$ where $d\ge k$. Divide $V$ into disjoint sets $S_1,...,S_k$, where 
    $$|S_i|=\begin{cases}
			\ceil{\frac{d}{k}}, & \text{if $i \le r$}\\
            \floor{\frac{d}{k}}, & \text{otherwise}
		 \end{cases}$$
    and let $W=\{w_1,...,w_{\ceil{\frac{d}{k}}+1}\}$ be disjoint from $V$.
    Construct a simplicial complex $MS_{d,k}'$, where $V$ is a face, as well as 
    all sets of the form $W\cup \bigcup_{i\ne j}S_i$ (for any $j$). 
    The induced subcomplex on $w_1,v_1,...,v_d$ is homotopic to $S^{k-1}$, as the maximal faces of this subcomplex are $\crb{w_1}\bigcup_{i\ne j}S_i$, and therefore the nerve is exactly $\partial(\Delta^k)$.\par
    Define $MS_{d,k}''$ by adding 
    the simplex on all vertices of $V\cup W\setminus \{w_1\}$.
    $MS_{d,k}''$ has a strongly connected $d$-skeleton, as the complete $d$-skeletons on $V\cup W\setminus \{w_1\}$ and $W\cup \bigcup_{i\ne j}S_i$ are obviously strongly connected, and $(W\setminus \{w_1\})\cup \bigcup_{i\ne j} S_i$-- the shared vertices among them-- have at least $d$ vertices in all cases. 
    Furthermore, $w_2,...,w_{\ceil{\frac{d}{k}}+1}$ are cone points, and thus the complex so far constructed is contractible.\par
    Finally, define $MS_{d,k}$ by attaching $(v_1,...,v_d,w_1)$ to $MS_{d,k}''$. The attachment is along $\partial(v_1,...,v_d,w_1)\subset MS_{d,k}''$, so by Mayer-Vietoris $H_k(MS_{d,k})\ne 0$.
    The map $i^*:H_k(sk_d(MS_{d,k}))\to H_k(MS_{d,k})$ is surjective since $d\ge k$, so $H_k(sk_d(MS_{d,k}))\ne 0$ as well.
    This complex has $|V|+|W|=d+\ceil{\frac{d}{k}}+1=\ceil{\frac{d(k+1)}{k}}+1$ vertices, proving tightness in \Cref{thmStr}.
\end{example}
    The above example is an explicit construction. It is also possible to provide an example motivating the similarity between the bounds of \Cref{thmPure} and \Cref{thmStr}:
\begin{example}\label{exKogSus}
    Let $X'$ be a vertex-minimal example of a pure $(d-1)$-\di al complex with $H_{k-1}(X')\ne 0$ (such as the one in \Cref{expure}), i.e $|X'|=\ceil{\frac{d(k+1)}{k}}$.
    Construct a new complex $CX'$ by adding a new vertex $v$, and for each $(d-1)$-face $f_i$ of $X'$, add a $d$-face $f_i\cup \{v\}$. 
    This complex is homeomorphic to $C(X')$- the cone above $X'$. Define $X''$ by attaching to $CX'$ a simplex along all vertices of $X'$, i.e $X''=\Delta^{|X'|-1}\cup_{X'} \mbox{st}(v)$ ($\mbox{st}(v)$ being the union of all faces containig $v$). $X''$ is easily seen to be homotopic to $\Sigma X'$.\par
    $H_k(\Sigma X')\ne 0$, and so $H_k(sk_d(X''))\ne 0$ as well ($k$ homology only depends on faces of \di \ up to $k+1$, so if $d>k$, taking the skeleton will not affect $H_k$, and if $k=d$ the number of generators may only increase). 
    $sk_d(X'')$ is also seen to be strongly connected as all $d$-faces are either in $sk_d(\Delta^{|X|-1})$ or attached to it along $d$ vertices.
\end{example}
The construction above first appears in \cite{kogan24}.

\section{Relative Strong Connectivity}\label{secRel}
\begin{claim}
    The complex $MH_{d,k}$ from \Cref{expure} is strongly connected w.r.t \di \ $d-\ceil{\frac{d-k}{k+1} }=d+1 -\ceil{\frac{d+1}{k+1}}=\floor{\frac{(d+1)k}{k+1}}$.
\end{claim}

\begin{proof}
    Recall that $MH_{d,k}$ is defined using vertex sets $s_0,...,s_{r}$ which contain $\ceil{\frac{d+1}{k+1}}$,  $s_{r+1},...,s_{k+1}$
    which contain $\floor{\frac{d+1}{k+1}}$, and $z$ which contains $1$ unless $k+1|d+1$ (where it is empty).
    $MH_{d,k}$ has $k+2$ maximal faces $f_0,...,f_{k+1}$, where $f_i=\bigcup_{j\ne i}s_j \bigcup_{j=i}^r z$.
    
    In the case where $k+1|d+1$, each pair of maximal faces share $\frac{(d+1)(k+2)}{k+1}-2\frac{d+1}{k+1}=\frac{(d+1)k}{k+1}$ vertices, and thus also strongly connected w.r.t this \di .
    Otherwise, for the maximal faces $f_i,f_j$, we divide into cases: 
    \begin{itemize}
        \item 
    If Both $i,j\le r$, 
    $$|f_i\cap f_j|=\ceil{\frac{(d+1)(k+2)}{k+1}}-|s_i|-|s_j|=\ceil{d+1+\frac{d+1}{k+1}}-2\ceil{\frac{d+1}{k+1}}=d+1+ \ceil{ \frac{d+1}{k+1}}-2\ceil{\frac{d+1}{k+1}}=$$
    $$d+1 -\ceil{\frac{d+1}{k+1}}=d+1 -\ceil{\frac{d+1+k-k}{k+1}}=d+1 -\ceil{ 1+ \frac{d-k}{k+1}}= d-\ceil{\frac{d-k}{k+1}}$$
    \item If $i,j>r$:
    $$|f_i\cap f_j|=\ceil{\frac{(d+1)(k+2)}{k+1}}-|s_i|-|s_j|-|z|=\ceil{d+1+\frac{d+1}{k+1}}-2\floor{\frac{d+1}{k+1}}-1=$$
    $$d+1+\floor{ \frac{d+1}{k+1}}+1-2\floor{\frac{d+1}{k+1}}-1=d+1-\floor{\frac{d+1}{k+1}}\ge d+1 -\ceil{\frac{d+1}{k+1}}$$
    \item Finally, if $i\le r<j$ (the other case is symmetric):
    $$|f_i\cap f_j|=\ceil{ d+1+\frac{d+1}{k+1}}-|s_i|-|s_j|-|z|=d+1+\ceil{ \frac{d+1}{k+1}}-\ceil{\frac{d+1}{k+1}}-\floor{\frac{d+1}{k+1}}-1=$$
    $$d+1-\floor{\frac{d+1}{k+1}}-1=d+1 -\ceil{\frac{d+1}{k+1}}$$
    \end{itemize}

    We conclude that every pair of maximal faces share $d+1 -\ceil{\frac{d+1}{k+1}}$ vertices, and that faces $f_i,i\le r$ are not in the same strongly connected component w.r.t \di s larger than $d+1 -\ceil{\frac{d+1}{k+1}}$.
\end{proof}

Now we address the case where the complex is strongly connected w.r.t a \di \ greater than $\floor{\frac{(d+1)k}{k+1}}$:

\begin{example}\label{exRel}
    Let $(d,k,m)$ be as in \Cref{thmRel}. Let $V=\crb{v_1,...,v_m}$, $W=\crb{w_{m+1},...,w_{d+1}}$, $Q=\crb{q_{1},...,q_{\ceil{\frac{m}{k}}}}$ be disjoint sets. 
    Divide $V$ into $k$ disjoint sets $V_i,\ 1\le i\le k$, where $|V_i|\in \crb{\floor{\frac{m}{k}},\ceil{\frac{m}{k}}}$ (this is possible in only a single way, up to permutation). Note $|Q|=max_i |V_i|$.
    Define $C'$ to be a complex on $V\cup W\cup Q$ such that the faces are:
    \begin{itemize}
        \item All $d$-\di al faces of $V\cup Q$ containing $q_1$.
        \item All $d$-\di al faces of $(\bigcup_{i\ne j} V_i) \cup W\cup Q$ containing $q_1$, for all $j$.
    \end{itemize}
    $(\bigcup_{i\ne j} V_i)\cup Q$ is common to both $V\cup Q$ and $(\bigcup_{i\ne j} V_i) \cup W\cup Q$, and has either $m$ or $m+1$ vertices by definition. Thus the complex is strongly connected w.r.t \di \ $m$. 
    Furthermore, $q_1$ is common to all maximal faces, and so is a cone point of $C'$, which is therefore contractible.\par
    On the other hand, the induced subcomplex on $V\cup W$ has maximal faces $V,\{(\bigcup_{i\ne j} V_i) \cup W\}_{j=1} ^k$ whose (maximal) nerve is $\partial(\Delta^k)$.
    Define $C$ by attaching a $d$-face along $V\cup W$ to $C'$.
    By Mayer–Vietoris we conclude that $H_k(C)\ne 0$.
    $C$ is also strongly connected w.r.t \di \ $m$ as the final face was attached along $V$, which has $m$ vertices.
    $$|C|=d+1+\ceil{\frac{m}{k}},$$
    proving tightness in \Cref{thmRel}.
\end{example}

\begin{proof}[Proof of \Cref{thmRel}]
    For any $i\in\N\cup\crb{0}$, let $F_i$ be the function sending a simplicial complex to its number of $i$-faces.
    Denote by $\mathfrak{C}_{-1}$ the set of (isomorphism types of) finite pure $d$-\di al complexes strongly connected w.r.t \di \ $m$ and having non-trivial $H_k$.
    For any $i\in\N\cup\crb{0}$, define $\mathfrak{C}_{i}:=F_i ^{-1}\prs{\min(F_i(\mathfrak{C}_{i-1}))}\cap \mathfrak{C}_{i-1}$. In other words, $\mathfrak{C}_{i}$ is the subset of complexes in $\mathfrak{C}_{i-1}$ having the minimal number of $i$-faces.
    An element in "$\mathfrak{C}_{\infty}$" (in our case it is enough to consider $\mathfrak{C}_{d}$) is said to be a \textbf{staggered dimensionwise minimal example} for the properties defining $\mathfrak{C}_{-1}$.
    $\mathfrak{C}_{-1}$ is non empty (by \Cref{exRel}) and the image of any $F_i$ always has a minimum. Therefore, $\mathfrak{C}_{d}$ is non-empty, and a staggered dimensionwise minimal example $C$ exists.
    
    $C$ is strongly connected w.r.t \di \ $m$, so we may chose a growth process $\Delta^d=C_0\subset...\subset C_n=C$.
    $C_{n-1}$ has fewer $d$-faces than $C$, so by staggered dimensionwise minimality $H_k(C_{n-1})=0$.
    If $C=C_{n-1}\cup_D \Delta^d$, by Mayer–Vietoris $H_{k-1}(D)\ne 0$, and thus by \Cref{lem:Nerve} $D$ has at least $k+1$ maximal faces with empty intersection- $D_0,...,D_k$ coming from maximal faces $D_0 ',...,D_k '$ of $X$. 
    $D$ also contains $m$ vertices which comprise a simplex (by strong connectivity), which we denote by $V=\{v_1,...,v_m\}$. Denote the other vertices by $W=\{w_{m+1},...,w_{d+1}\}$. Applying \Cref{lem:Nerve} again we may assume WLOG that $V\subset D_0$.\par 
    If $|D_i\cap V|>\frac{m(k-1)}{k}$ for all $i$, by the pigeonhole principle there must be a vertex common to all $D_i$, contradicting our choice of $D_i$. 
    Therefore, $min_i|D_i\cap V|\le \frac{m(k-1)}{k}$ (assume WLOG that this minimum is obtained with $D_1$), and:
    $$|D_1 '\setminus D|\ge |D_1'|-|D_1\cap V|-|D\setminus V|\ge d+1-\frac{m(k-1)}{k}-(d+1-m)=\frac{m}{k}.$$
    Since $C$ contains both $D_1 '\setminus D$ and the $d$-face attached along $D$, and they are disjoint, this gives the desired bound. Notice that this recreates the results in \Cref{thmPure} and \Cref{thmStr} for appropriate choices of $m$.

    By \Cref{exRel} the bound is also tight.
\end{proof}
  
\section{Discussion}\label{secDiscuss}
The importance of strong connectivity to topology is seen in the following:

\begin{claim}\label{lemCupSteenConc}  (proved in \cite[Lemma 2.7]{kogan24})
  For a simplicial complex $X$: $H^d(X)$,  $Im (\vee : H^*(X)\times H^*(X)\to H^d(X))$  and steenrod operations landing in $H^d(X)$ are non-zero only if there exists a $d$-\di al strong-connectivity component on which they are non-zero. 
        If $X$ has $d$ as the maximal \di , this is an if and only if for $H^d(X)$.
\end{claim}
A corollary to \Cref{lemCupSteenConc} is that \Cref{thmStr} my be used as a lower bound for any simplicial complex where the image of the cup product or some steenrod operation in $H^d$ is non-trivial.\par
This should be compared to the bound in \cite[\textsection 16]{ArnouxMarin91}, which says that if a complex $X$ has non-trivial $\cup:H^k(X)\times H^{d-k}(X)\to H^d(X)$, then $|X|\ge d+1+max(k,d-k)+1$ (by having vertex disjoint $\Delta^d$ and $\Delta^{max(k,d-k)}$ as subcomplexes).
Of course, $d+1+max(k,d-k)+1$ is much bigger than $d+1+\frac{d}{k}$ in all cases except $k=1$.

\vspace{12pt}
The importance of relative strong connectivity is seen from its influence on collapsability:

\begin{claim}\label{lemColapseComponents}(to appear in a different article)
    A simplicial complex $C$ collapses onto \di \ $d$ if and only if all strong connectivity components of $C$ w.r.t \di\  $d+1$ collapse onto \di \ $d$.
\end{claim}

Regarding the theorems in this paper-- while they do not take into account the whole Betti-vector, they do produce an upper bound on the minimal number of vertices needed:
\begin{example}\label{exMore}
    let $b=(b_1,b_2,...,b_k)$ be a list of non-negative integers, and $d\ge k$. Then there exists a pure $d$-\di al (or strongly connected w.r.t \di \ $d$) complex with a number of vertices behaving asymptotically (in $d,b_1,b_2,...,b_k$) as 
    $$\sum_{i=1}^k\sqrt[i+2]{b_i}\frac{(d+1)(i+2)}{i+1}\gamma_i -(d+1)(|I|-1),$$
    or in the strongly connected case as
    $$\sum_{i=1} ^k \sqrt[i+1]{b_i}\frac{(d+1)(i+1)}{i}\gamma_{i-1} -(d+1)(|I|-1).$$
    The constant $\frac{1}{e}\le\gamma_i\le 1$ stems from Stirling's approximation, and $I$ is the set of indexes for which $b_i\ne 0$. 
    These are obtained by repeating the constructions in \Cref{expure} (\Cref{exKogSus}) for a larger "basic cycle" of $sk_i(\Delta^n)$ ($sk_{i-1}(\Delta^n)$ and then suspending), with $n$ determined by $b_i$, and then gluing the resulting complexes along a $d$-face.
\end{example}
Compare \Cref{exMore} to \cite[Theorem 1.3]{bjornerkalai88} which states that a Betti-vector $b=(b_1,b_2,...,b_k)$ is realizable using $n+1$ vertices if and only if $b$ is the face vector of a Sperner family (anti-chain) of subsets of $[n]$. \par
 The Lubell–Yamamoto–Meshalkin inequality states that a Sperner family with a face-vector $(a_0,...,a_n)$ supported on $[n]$ must satisfy 
 $$\sum_{i=0} ^n a_i i!(n-i)!\le n!$$
 We thus see that $S=\sum_{i=1} ^n i!(n-i)!b_i\le n!$. Asymptotically approximating (following \cite[formula 67]{Borwein17}):
 $$n\approx\Gamma^{-1}(n!)=\theta\left (\frac{ln(n!/\sqrt{2\pi})}{W_0(\frac{1}{e}ln(\frac{n!}{\sqrt{2\pi}})}\right)$$
  $W_0$ being the (main branch of) Lambert W function. Substituting the approximation from \cite[4.11]{CGHJK96} we get:
  $$n\approx\Gamma^{-1}(n!)=\theta\left (\frac{ln(n!/\sqrt{2\pi})}{ln(\frac{1}{e}ln(\frac{n!}{\sqrt{2\pi}})-lnln(\frac{1}{e}ln(\frac{n!}{\sqrt{2\pi}})}\right)=\theta\left (\frac{ln(n!/\sqrt{2\pi})}{ln(\frac{1}{e}ln(\frac{n!}{\sqrt{2\pi}})}\right)= \Omega\left (\frac{ln(S/\sqrt{2\pi})}{ln(\frac{1}{e}ln(S)}\right)$$
  and we get $n=\Omega(\frac{ln(S)}{ln(ln(S))})$.
 While the LYM inequality is not tight in non-trivial cases, if $n=o(S^\epsilon),\epsilon>0$ for minimal examples, we may conclude that the requirements of purity of \di \ and strong connectivity seem to require an asymptotically greater number of vertices.
 
\printbibliography 

@book{AT,
  author = {Hatcher, Allen},
  isbn = {0-521-79540-0},
  publisher = {Cambridge University Press},
  title = {Algebraic Topology},
  year = {2002}
}

@misc{kogan24,
      title={Multi-Dimensional Cohomological Phenomena in the Multiparametric Models of Random Simplicial Complexes}, 
      author={Jon V. Kogan},
      year={2024},
      eprint={2402.02573},
      archivePrefix={arXiv},
      primaryClass={math.AT},
      url={https://arxiv.org/abs/2402.02573}, 
}

@article{bjornerkalai88,
  title={An extended Euler-Poincar{\'e} theorem},
  author={Bj{\"o}rner, Anders and Kalai, Gil},
  year={1988},
journal={Acta Mathematica}
}

@article{ArnouxMarin91,
  title={The Kuhnel triangulation of the complex projective plane from the view point of complex crystallography Part I},
  author={Pierre Arnoux and Alexis Marin},
  journal={Memoirs of The Faculty of Science, Kyushu University. Series A, Mathematics},
  year={1991},
  volume={45},
  pages={55-142},
  url={https://api.semanticscholar.org/CorpusID:120299778}
}

@article{Borwein17,
  title={Gamma and Factorial in the Monthly},
  author={Jonathan Michael Borwein and Robert M Corless},
  journal={The American Mathematical Monthly},
  year={2017},
  volume={125},
  pages={400 - 424},
  url={https://api.semanticscholar.org/CorpusID:119324101}
}

@article{CGHJK96,
author = {Corless, Robert and Gonnet, Gaston and Hare, D. and Jeffrey, David and Knuth, D.},
year = {1996},
month = {01},
pages = {329-359},
title = {On the Lambert W Function},
volume = {5},
journal = {Advances in Computational Mathematics},
doi = {10.1007/BF02124750}
}

@article{Kruskal63,
  title={The number of simplices in a complex},
  author={Kruskal, Joseph B},
  journal={Mathematical optimization techniques},
  volume={10},
  pages={251--278},
  year={1963}
}
\end{document}